\documentclass[a4paper,12pt]{article}
\usepackage{geometry,latexsym,amssymb}
\usepackage[latin5]{inputenc}
\usepackage{enumerate}
\usepackage{enumitem}
\usepackage[usenames]{color}
\usepackage{graphicx}
\usepackage{enumerate}
\usepackage{rotating}
\usepackage{multirow}
\usepackage{cite}
\usepackage{amsthm}
\usepackage{caption}
\usepackage{tikz}
\usetikzlibrary{matrix,arrows}
\usepackage{amsmath,amscd}
\usepackage{calc}
\usepackage{tabularx}
\usepackage{ifthen}
\usepackage{eucal}
\usepackage{amssymb,amscd,amsthm, verbatim,amsmath,color,fancyhdr, mathrsfs}
\usepackage{graphicx}
\usepackage{turnstile}
\usepackage{hyperref}
\usepackage[colorinlistoftodos]{todonotes}

\newtheorem{thm}{Theorem}[section]
\newtheorem{prop}[thm]{Proposition}
\newtheorem{lem}[thm]{Lemma}
\newtheorem{rem}[thm]{Remark}

\newtheorem{defn}[thm]{Definition}

\begin{document}

\begin{center}
{\Large \textbf{Cauchy numbers in type $B$}} \vspace*{0.5cm}
\end{center}

\begin{center}\begin{center}
Alnour Altoum$^{a,1}$,  Hasan Arslan$^{b,2}$, Mariam Zaarour$^{a,3}$ \\
$^{a}${\small {\textit{Graduate School of Natural and Applied Sciences, Erciyes University, 38039, Kayseri, Turkey}}}\\
$^{b}${\small {\textit{Department of  Mathematics, Faculty of Science, Erciyes University, 38039, Kayseri, Turkey}}}\\
{{\small {\textit{$^{1}$alnouraltoum178@gmail.com}}}\small {\textit{ $^{2}$hasanarslan@erciyes.edu.tr}}}\\
{\small {\textit{$^{3}$mariamzaarour94@gmail.com }}}\\[0pt]
\end{center}
\end{center}

\begin{abstract}
In this paper, we will introduce the Cauchy numbers of both kinds in type B and produce their corresponding exponential generating functions. Then we will provide some identities involving Cauchy, Lah, and Stirling numbers in type B through combinatorial methods.
\end{abstract}


\\

Cauchy numbers are defined as the integration of rising and falling factorials \cite{comtet1974advanced}. These numbers can be classified into two categories: Cauchy numbers of the first kind, denoted by $C_n$,  and Cauchy numbers of the second kind, denoted as $c_n$. The first kind is defined by
\begin{equation*}
	C_n = \int_{0}^{1} (x)_n \, dx 
\end{equation*}
where $(x)_n := x(x-1)(x-2) \cdots (x-n+1)$ represents the falling factorial. The second kind is defined to be
\begin{equation*}
	c_n = \int_{0}^{1} [x]_n \, dx 
\end{equation*}
where $[x]_n :=x(x+1)(x+2)(x+3) \cdots (x+n-1)$ represents the rising factorial \cite{comtet1974advanced}.

The ordinary generating function of any  infinite sequence $(a_k)_{k \in \mathbb{N}}$ is defined by $g(x)=\sum_{k=0}^{\infty}a_k x^k$ and shortly denoted by $\mathcal{G}(a_k)=g(x)$. The corresponding exponential generating function to the sequence $(a_k)_{k \in \mathbb{N}}$ has the form  $\hat{g}(x)=\sum_{k=0}^{\infty}\frac{a_k}{k!} x^k$ and briefly denoted by $\mathcal{E}(a_k)=\hat{g}(x)$. Due to \cite{merlini2007}, the inverse operator of a formal power series $g(x)=\sum_{k=0}^{\infty}a_k x^k$ provides the coefficient of $x^k$ such that $[x^k]g(x)=a_k$ for all $k \in \mathbb{N}$.

Riordan array is actually a direct coefficient derivation method and it was first introduced by Shapiro in \cite{shapiro1991}. Thus, it is useful to derive the exponential generating function of Cauchy numbers. Riordan array $B=(b_{n,k})_{n,k \in \mathbb{N}}$, which is a lower triangular and infinite matrix, is identified with a pair of formal power series such that $B=\mathcal{R}(b_{n,k})=(b(x),c(x))$, where 
\begin{equation}\label{riordan0}
	b_{n,k}=[x^n]b(x)(xc(x))^k
\end{equation}
for all $n \in \mathbb{N}$ where $\mathbb{N}$ denotes the set of natural numbers. One of the fundamental properties of Riordan array is the \textit{summation property} which is given as follows:
\begin{equation}\label{riordan1}
	\sum_{k=0}^{n}b_{n,k}g_k=[x^n]b(x)g(xc(x))
\end{equation}
where $\mathcal{R}(b_{n,k})=(b(x),c(x))$ and $g(x)$ is the ordinary generating function of the sequence $(g_k)_{k \in \mathbb{N}}$. We will mainly use the following expression instead of Eq. (\ref{riordan1}):
\begin{equation}\label{riordan2}
	\sum_{k=0}^{n}b_{n,k}g_k=[x^n]b(x)\left [g(y) : y=xc(x) \right ].
\end{equation}
It is well-known from \cite{merlini2006cauchy} that,
$$\mathcal{R}(\frac{k!}{n!}c(n,k))=\left(1, \frac{1}{x} ln\frac{1}{1-x} \right)~\textrm{and}~\mathcal{R}(\frac{k!}{n!}S(n,k))=\left(1, \frac{e^x-1}{x}  \right).$$
where $c(n, k)$ is the classical signless Stirling number of the first kind and $S(n,k)$ is the classical Stirling numbers of the second kind. 

According to \cite{merlini2006cauchy}, the classical Cauchy numbers of the first kind have the exponential generating function which is given by:
\begin{equation*}
	\mathcal{E}(C_n)=\sum_{k=0}^{\infty} C_n \frac{x^n}{n!} = \frac{x}{ln(1+x)}
\end{equation*}
and these numbers are related to the Stirling numbers of the first kind through the formula:
\begin{equation*}
	C_n = \sum_{k=0}^{n} \frac{ s(n,k)}{k+1}
\end{equation*}
where $s(n, k):= (-1)^{n-k}c(n, k)$ is known as a Stirling number of the first kind. The exponential generating function of Cauchy numbers of the second kind has the following form (see \cite{merlini2006cauchy}):
\begin{equation*}
		\mathcal{E}(c_n)=\sum_{k=0}^{\infty} c_n \frac{x^n}{n!} = \frac{x}{(1+x)ln(1+x)}
\end{equation*}
and these numbers can be expressed in terms of signless Stirling numbers of the first kind as follows: 
\begin{equation*}
	c_n = (-1)^n \sum_{k=0}^{n}  \frac{c(n,k) }{k+1}.
\end{equation*}

The rest of this paper is organized as follows: In section 2, we recall the concept of Stirling numbers in type $B$. In section 3, we introduce the Cauchy numbers of both kinds and Lah numbers in type B. Furthermore, we drive the exponential generating functions of Cauchy numbers in type $B$ with the help of the Riordan array. Finally, we generalize the Cauchy and Lah numbers into $G_{m,n}$ type.

\section{Preliminaries}
The Stirling numbers of the second kind in type $B$ which is denoted by $S_B(n,k)$ was defined first by Reiner in \cite{br2} by the following recurrence relation:
\begin{equation*}
S_B(n,k)=S_B(n-1,k-1)+(2k+1)S_B(n-1,k), \quad 1 \leq k < n 
\end{equation*}
with the initial conditions  $S_B(n,n)=S_B(n,0) = 1$ for all $n \geq 0$.
The signless Stirling numbers of the first kind in type $B$ are identified with the recurrence relation 
\begin{equation*}
	c_B(n,k)=c_B(n-1,k-1)+(2n-1) c_B(n-1,k), \quad   k\geq 0
\end{equation*}
where  $c_B(n,n)=c_B(1,0) =1$ and $c_B(n,k)=0$ if $k<0$ (see  \cite{sagan2022q}). We also note that the number $s_B(n, k):=(-1)^{n-k}c_B(n, k)$ is known as a Stirling number of the first kind. The second kind of Stirling numbers in type $B$ corresponds to the sequence oeis.org/A039755 in OEIS. One could see Stirling numbers in type $B$ of the second kind for small values of $n$ and $k$ in Table \ref{tab:table3}.
\newpage
\begin{table}[h!] 
\caption{Second kind Stirling numbers in type $B$}
	\label{tab:table3}		
	\centering
	\scalebox{0.9}{
			\begin{tabular}{c|cccccccc}
					\hline
					n & $S_B(n,0)$ & $S_B(n,1)$ &$S_B(n,2)$&$S_B(n,3)$& $S_B(n,4)$& $S_B(n,5)$&$S_B(n,6)$& $S_B(n,7)$\\
					\hline
					0&1\\
					1&1&1\\
					2&1&4&1\\
					3&1&13&9&1\\
					4&1&40&58&16&1\\
					5&1&121&330&170&25&1\\
					6&1&364&1771&1520&395&36&1\\
					7&1&1093&9219&12411&5075&791&49&1\\
					\hline
			\end{tabular}}
\end{table}

The following table presents some Stirling numbers of the first kind $c_B(n,k)$ which is associated with the sequence oeis.org/A039758 in OEIS.

\begin{table}[h!]
	\caption{Fist kind signless Stirling numbers in type $B$}
	\label{tab:table4}		
	\centering
	\scalebox{0.9}{
			\begin{tabular}{c|cccccccc}
					\hline
					n & $c_B(n,0)$ & $c_B(n,1)$ &$c_B(n,2)$&$c_B(n,3)$& $c_B(n,4)$& $c_B(n,5)$&$c_B(n,6)$& $c_B(n,7)$\\
					\hline
					0&1\\
					1&1&1\\
					2&3&4&1\\
					3&15&23&9&1\\
					4&105&176&86&16&1\\
					5&945&1689&950&230&25&1\\
					6&10395&19524&12139&3480&505&36&1\\
					7&135135&264207&177331&57379&10045&973&49&1\\
					\hline
			\end{tabular}}
\end{table}

The following theorem, which is provided by \cite{bagno2019some} and \cite{sagan2022q}, expresses $x^n$ as a sum of the terms involving both Stirling numbers of the second kind and falling factorial in type $B$. The falling factorial in type $B$ is defined to be $(x)_n^B = (x-1)(x-3) (x-5) \cdots (x-2n+1)$ with initial condition $(x)^B_0 = 1$ (see  \cite{sagan2022q}).

\begin{thm} \label{m}
	For any integer $n\geq 0$, we have
\begin{equation*}
x^n = \sum_{k=0}^{n} S_B(n,k)(x)_k^B.
\end{equation*}
\end{thm}
Taking into account Corollary 2.7 in \cite{sagan2022q}, it is observed that the matrices $[s_B(n,k)]_{n,k\geq0}$  and $[S_B(n,k)]_{n,k\geq0}$ are inverse of each other. Therefore, it is easy to see that 
\begin{equation}\label{9}
	(x)_n^B = \sum_{k=0}^{n} s_B(n,k)x^k.
\end{equation}

The rising factorial in type $B$ is defined as $[x]^B_n = (x+1)(x+3)(x+5) \cdots (x+2n-1)$  with the initial condition $[x]^B_0 = 1$. It is well-known from part (c) of Theorem 2.1 in \cite{sagan2022q} that for any nonnegative integer n
\begin{equation} \label{50}
[x]^B_n = \sum_{k=0}^{n} c_B(n,k) x^k.
\end{equation}
	
Conversely, for all $n\in \mathbb{N}$ the ordering powers $x^n$ can be easily expressed as a linear combination of rising factorials $[x]_n^B$ as follows:
\begin{equation} \label{7}
x^n = \sum_{k=0}^{n} S_B(n,k) (-1)^{n-k} [x]^B_k ~~\textrm{for}~\textrm{all}~ n \geq 0.
\end{equation}

\section{Cauchy numbers in type $B$}
\label{sec:sample1}
In this section, we will introduce the notions of Cauchy and Lah numbers in type $B$. These numbers will be defined using both falling and rising factorials in type $B$. Additionally, we will explore some relationships between Cauchy, Stirling, and Lah numbers.
\begin{defn}
	The type $B$ Cauchy numbers of the first kind are defined by the following definite integral
	\begin{equation*}
		C_n^B = \int_{0}^{1} (x)^B_n \,dx.
	\end{equation*}
\end{defn}

Table \ref{tab:table1} records some values of the first kind of Cauchy numbers by giving a few small $n$ values.

\begin{table}[h!]
	\caption{Cauchy numbers of the first kind in type $B$}\label{2}
	\label{tab:table1}		
	\centering
	\begin{tabular}{|c|c|c|c|c|c|c|c|c|}
		\hline
		n & 0&1 &2 & 3 & 4& 5& 6 & 7\\
		\hline
		$C_n^B$ &1& -1/2 & 4/3& -25/4 &628/15&-729/2& 81994/21&-1191619/24\\
		\hline
	\end{tabular}
\end{table}

\begin{prop}
The Cauchy numbers of the first kind hold for the following recurrence relation 
\begin{equation*}
C_{n+1}^B +  (2n+1) C_n^B= \sum_{k=0}^{n} \frac{s_B(n,k)}{k+2}.
\end{equation*}

\end{prop}
\begin{proof}
Due to the definition of the falling factorial of type $B$, we can write the relation $x(x)_n^B = (x)_{n+1}^B + (2n+1)(x)_n^B$. Therefore, by Theorem \ref{m} we get 
			\begin{align*}
				C_{n+1}^B &= \int_{0}^{1} (x)_{n+1}^B \,dx = \int_{0}^{1} (x(x)_n^B - (2n+1)(x)_n^B) \,dx \\ 
				&= \int_{0}^{1} \sum_{k=0}^{n} s_B(n,k) x^{k+1}  \,dx - (2n+1) \int_{0}^{1} (x)_{n}^B \,dx\\ 
				&= \sum_{k=0}^{n} \frac{s_B(n,k)}{k+2} - (2n+1) C_n^B,
			\end{align*}
			as desired.
\end{proof}
  
\begin{defn}
Cauchy numbers of the second kind in type $B$ are defined by definite integral as below:
\begin{equation*}
c^B_n = \int_{0}^{1} [x]^B_n\,dx.
\end{equation*}
\end{defn}

Table \ref{tab:table2} displays some special values for the second kind of Cauchy numbers of type $B$.
\begin{table}[h!]
\caption{Cauchy numbers of the second kind in type $B$}\label{2}
\label{tab:table2}		
	\centering
	\begin{tabular}{|c|c|c|c|c|c|c|c|c|}
		\hline
		n & 0&1 &2 & 3 & 4& 5& 6 & 7\\
		\hline
		$c_n^B$ &1& 3/2 & 16/3& 119/4 &3388/15&13013/6& 528790/21&2742975/8\\
		\hline
\end{tabular}
\end{table}
\begin{prop}
We have the following recurrence relation for the Cauchy numbers of the second kind in type $B$:
	
\begin{equation*}
\quad c_{n+1}^B -(2n+1)c_n^B = \sum_{k=0}^{n} \frac{c_B(n,k)}{k+2}.
\end{equation*}
\end{prop}
\begin{proof}
We can deduce the relation $[x]^B_{n+1}=x[x]^B_n+(2n+1)[x]^B_n$ from the definition of the rising factorial of type $B$. Therefore, the desired result can be easily seen from Eq. (\ref{50}).
\end{proof}

\begin{thm} \label{n}
For any positive integer $n$, we have the following formula
\begin{equation*}
\sum_{k=0}^{n} S_B(n,k) C_k^B = \frac{1}{n+1}= \sum_{k=0}^{n} S_B(n,k) (-1)^{n-k} c_k^B. 
\end{equation*}
\end{thm}
\begin{proof}
Considering Theorem \ref{m} and the definition of the first kind of Cauchy numbers in type $B$, we have 
\begin{equation*}
\sum_{k=0}^{n} S_B(n,k) C_k^B = \sum_{k=0}^{n} S_B(n,k)\int_{0}^{1} (x)_n^B \,dx =  \int_{0}^{1} x^n \,dx = \frac{1}{n+1}.
\end{equation*}
Using Eq. (\ref{7}) and the definition of the second kind of Cauchy numbers of type B, then we get
\begin{equation*}
\sum_{k=0}^{n} S_B(n,k) (-1)^{n-k}c_k^B  =  \sum_{k=0}^{n} S_B(n,k) (-1)^{n-k}\int_{0}^{1} [x]_k^B \,dx= \int_{0}^{1} x^n \,dx = \frac{1}{n+1}.   
\end{equation*}
\end{proof}

 \subsection{The exponential generating functions of Cauchy numbers in type $B$}
 In order to derive the exponential generating functions for Cauchy numbers of both kinds in type $B$, we will apply the Riordan array of the signless Stirling numbers of the first kind in type $B$. Before going into a further discussion of the exponential generating functions, we will give an important relationship between Cauchy numbers of both kinds and Stirling numbers of the first kind in type $B$.
\begin{lem} \label{5}
	For all $n \geq 0$, we have the following relations:
	\begin{itemize}
		\item [1.] $C_n^B = \sum_{k=0}^{n}  \frac{s_B(n,k)}{k+1}$,
		\item[2.] $c_n^B = \sum_{k=0}^{n} \frac{c_B(n,k)}{k+1}$ .
	\end{itemize}
\end{lem}
\begin{proof}
Considering Eq. (\ref{9}), we then have 
\begin{align*}
C_n^B &=  \int_{0}^{1} (x)_n^B \,dx  \\ 
&=  \sum_{k=0}^{n} s_B(n,k) \int_{0}^{1} x^k \,dx \\
&= \sum_{k=0}^{n}  \frac{s_B(n,k)}{k+1}. 
\end{align*}
From Eq. (\ref{50}), we immediately obtain the second part as   
\begin{equation*}
c_n^B = \int_{0}^{1} [x]_n^B \,dx = \sum_{k=0}^{n} c_B(n,k) \int_{0}^{1} x^k \,dx = \sum_{k=0}^{n} \frac{c_B(n,k)}{k+1}. 
\end{equation*}
\end{proof}

We are now in a position to give the exponential generating functions of Cauchy numbers of both kinds in type $B$.

\begin{thm}\label{g1}
For the Cauchy numbers of the second kind, we have
	$$\mathcal{E}(c_n^B)=\sum_{n \geq 0}c_n^B~\frac{x^n}{n!}=\frac{1-\sqrt{1-2x}}{(2x-1)ln\sqrt{1-2x}}.$$
\end {thm}
	
\begin{proof}
We can write $\sum_{n \geq 0}c_B(n,k)~\frac{x^n}{n!}=\frac{1}{k!\sqrt{1-2x}}\left(ln \frac{1}{\sqrt{1-2x}}\right)^k$ from Theorem 4.1 (c) in \cite{sagan2022q}. Therefore, we can deduce from Eq. (\ref{riordan0}) that
\begin{equation} \label{riordan3}
\mathcal{R}(\frac{k!}{n!}c_B(n,k))=\left(\frac{1}{\sqrt{1-2x}},\frac{1}{x} ln\frac{1}{\sqrt{1-2x}} \right).
\end{equation}
Using the second part of Lemma \ref{5}, we can get $\frac{c_n^B}{n!}=\sum_{k=0}^n \frac{k!}{n!}c_B(n,k)\frac{1}{(k+1)!}$. Since the ordinary generating function of $\frac{1}{(k+1)!}$ is equal to $\frac{e^x-1}{x}$ and considering Eq. (\ref{riordan2}) and Eq. (\ref{riordan3}) , we then extract
$$\frac{c_n^B}{n!}=[x^n]\frac{1}{\sqrt{1-2x}}\left [\frac{e^y-1}{y} : y= ln\frac{1}{\sqrt{1-2x}} \right ],$$
		as desired.
\end{proof}

\begin{thm} \label{g2}
For the Cauchy numbers of the first kind, we have 
$$\mathcal{E}((-1)^{(n-1)}C_n^B)=\sum_{n \geq 0}(-1)^{(n-1)}C_n^B~\frac{x^n}{n!}=\frac{1-\sqrt{1-2x}}{\sqrt{1-2x}~~ln\sqrt{1-2x}}.$$
\end {thm}
		
\begin{proof}
Using the first part of Lemma \ref{5}, we can write $\frac{C_n^B}{n!}=(-1)^{n-1} \sum_{k=0}^n \frac{k!}{n!}c_B(n,k)\frac{(-1)^{k+1}}{(k+1)!}$. Since the ordinary generating function of $\frac{(-1)^{k+1}}{(k+1)!}$ is equal to $\frac{e^{-x}-1}{x}$ and taking into consideration Eq. (\ref{riordan2}) and Eq. (\ref{riordan3}), we then derive the desired formula by means of the following relation:
$$\frac{C_n^B}{n!}=[x^n]\frac{1}{\sqrt{1-2x}}\left [\frac{e^{-y}-1}{y} : y= ln\frac{1}{\sqrt{1-2x}} \right ].$$
			
\end{proof}

One can illustrate Theorem \ref{g1} and Theorem \ref{g2} by considering Table \ref{tab:table2} and \ref{tab:table1}, respectively.

\subsection{Lah numbers in type $B$}
The classical Lah numbers, which were discovered by Ivo Lah in 1954 (see \cite{lah1954new}), are defined by binomial coefficient as
\begin{equation*}
 L(n,k) = \frac{n!}{k!} \binom{n-1}{k-1},
\end{equation*}
and were also defined by means of the Stirling numbers as follows (see \cite{lindsay2011new}):
\begin{equation*} 
L(n,k) = \sum_{j=k}^{n} c(n,j) S(j,k)  
\end{equation*} 
where $c(n,j)$ and $S(j,k)$ are the classical Stirling numbers of the first and second kind, respectively. The recurrence relation of $L(n,k)$ is given by
\begin{equation*}
L(n,k) = L(n-1,k-1) + (n-1+k) L(n-1,k) ~~\textrm{for}~\textrm{all}~~ n,k \in \mathbb{N}.
\end{equation*}
It is well-known from \cite{lindsay2011new} that the exponential generating function of $L(n,k)$ is stated as
 \begin{equation*}
\mathcal{E}(L(n,k))=\sum_{n \geq 0}L(n,k) ~\frac{x^n}{n!} = \frac{1}{k!}( \frac{x}{1-x})^{k}.
 \end{equation*}
Thus, we can conclude that Riordon array $\mathcal{R}(\frac{k!}{n!}L(n,k))=\left(1,\frac{1}{1-x}\right).$
\begin{defn}
The Lah numbers in type $B$ may be defined by 
  \begin{equation*}
      L_B(n,k) = \binom{n}{k}^2 2^{n-k}(n-k)!
  \end{equation*}
and the type $B$ Lah numbers can be expressed as a linear combination of Stirling numbers:
\begin{equation*}
L_B(n,k) = \sum_{j=k}^{n} c_B(n,j) S_B(j,k) ~~\textrm{for}~\textrm{all}~~ n,k \in \mathbb{N}.
\end{equation*}
\end{defn} 
The recurrence relation of $L_B(n,k)$ is defined by 
\begin{equation}\label{recurrence}
L_B(n,k) = L_B(n-1,k-1) + 2(n+k) L_B(n-1,k) ~~\textrm{for}~\textrm{all}~~ n,k \in \mathbb{N},
\end{equation}
with the initial conditions $L_B(n,0) = 2^nn!$, $L_B(n,n) =1$ and $L_B(n,k) = 0$ if $k<0$.

\begin{thm}
The exponential generating function of $L_B(n,k)$ is given by 
\begin{equation} \label{egfL_B}
\mathcal{E}(L_B(n,k))=\sum_{n \geq 0}L_B(n,k) ~\frac{x^n}{n!} = \frac{x^{k}}{(1-2x)^{k+1}k!}.
\end{equation}  
\end{thm}
\begin{proof}
We argue by induction on $k$. If $k=0$, then we get  
$$\sum_{n \geq 0}L_B(n,0) ~\frac{x^n}{n!} = \sum_{n \geq 0} 2^n n!~\frac{x^n}{n!} = \frac{1}{1-2x}.$$
Now let $ \mathcal{E}(L_B(n,k)):=f_k(x)$. Thus we can write $f_k(x) = \sum_{n \geq 0}L_B(n,k) ~\frac{x^n}{n!}$. Considering Eq. (\ref{recurrence}), we conclude that
\begin{align*}
    f_k(x) &:= \sum_{n \geq k}L_B(n,k) ~\frac{x^n}{n!}\\
        &=\sum_{n \geq k}L_B(n-1,k-1) ~\frac{x^n}{n!} + \sum_{n \geq k+1}2(n+k)L_B(n-1,k) ~\frac{x^n}{n!}\\
    &= \sum_{n \geq k-1}L_B(n,k-1) ~\frac{x^{n+1}}{(n+1)!} + \sum_{n \geq k}2(n+1+k)L_B(n,k) ~\frac{x^{n+1}}{(n+1)!} \\
    &= \sum_{n \geq k-1}L_B(n,k-1) ~\frac{x^{n+1}}{(n+1)!} +2 \sum_{n \geq k}L_B(n,k) ~\frac{x^{n+1}}{n!} + 2k \sum_{n \geq k}L_B(n,k) ~\frac{x^{n+1}}{(n+1)!} \\
    &=  \sum_{n \geq k-1}L_B(n,k-1) ~\frac{x^{n+1}}{(n+1)!} +2x f_k(x) + 2k  \sum_{n \geq k}L_B(n,k) ~\frac{x^{n+1}}{(n+1)!} 
\end{align*}
hence 
\begin{equation*}
    (1-2x)f_k(x) = \sum_{n \geq k-1}L_B(n,k-1) ~\frac{x^{n+1}}{(n+1)!} + 2k  \sum_{n \geq k}L_B(n,k) ~\frac{x^{n+1}}{(n+1)!}. 
\end{equation*}
If we take the derivative of both sides of the previous equation with respect to $x$, then we get 
\begin{equation*}
    (1-2x)f^{\prime}_k (x) -2f_k(x) = f_{k-1}(x) + 2kf_k(x) 
\end{equation*}
hence
\begin{equation*}
      (1-2x)f^{\prime}_k (x) - 2(k+1)f_k(x) = f_{k-1}(x).
\end{equation*}
Multiplying both sides of the above equality by $(1-2x)^k$, then we obtain by the induction hypothesis for $f_{k-1}(x)$ that
\begin{equation*}
    (1-2x)^{k+1}f^{\prime}_k (x) - 2(k+1)(1-2x)^kf_k(x) = \frac{x^{k-1}}{(k-1)!}.
\end{equation*}
Thus we have
\begin{equation*}
     ((1-2x)f_k(x))^{\prime} = \frac{x^{k-1}}{(k-1)!}.
\end{equation*}
If we solve the initial value problem $((1-2x)f_k(x))^{\prime} = \frac{x^{k-1}}{(k-1)!},~~f_k(0)=0$, we
then obtain
\begin{equation*}
f_k(x) = \frac{x^{k}}{(1-2x)^{k+1}k!},
\end{equation*}
as desired.
\end{proof}

It is straightforward to check that Eq. (\ref{egfL_B}) gives the Riordon array of $\frac{k!}{n!}L_B(n,k)$ as
$$\mathcal{R}(\frac{k!}{n!}L_B(n,k))=\left(\frac{1}{1-2x},\frac{1}{1-2x}\right).$$
Due to \cite{kim2020lah}, the classical Lah-Bell numbers $L(n)$ are defined as
\begin{equation*}
    L(n) = \sum_{k=0}^{n} L(n,k)
\end{equation*}
and the exponential generating function of classical Lah-Bell numbers is given by
\begin{equation*}
    \mathcal{E}(L(n))=\sum_{n=0}^{\infty} L(n) \frac{x^n}{n!} = e^{\frac{x}{1-x}}.
\end{equation*}
Therefore, we can define Lah-Bell numbers in type $B$ as
\begin{equation*}
    LB(n) = \sum_{k=0}^{n} L_B(n,k). 
\end{equation*}
From Eq. (\ref{egfL_B}), we can immediately deduce the exponential generating function of $LB(n)$ in the following form:
\begin{equation*}
    \mathcal{E}(LB(n))=\sum_{n=0}^{\infty} LB(n) \frac{x^n}{n!} = (\frac{1}{1-2x}) e^{\frac{x}{1-2x}}.
\end{equation*}
In Table \ref{tab:table4}, one could see both Lah numbers and Lah-Bell numbers of type $B$ for small values of $n$ and $k$.

		\begin{table}[h!]
  \caption{Lah numbers and Lah-Bell numbers $LB(n)$ of type $B$}\label{tab:table4}

			\centering
   \scalebox{0.8}{
			\begin{tabular}{c|c||cccccccc}
				\hline
				n & $LB(n)$&$L_B(n,0)$ & $L_B(n,1)$ &$L_B(n,2)$&$L_B(n,3)$& $L_B(n,4)$& $L_B(n,5)$&$L_B(n,6)$& $L_B(n,7)$\\
				\hline
				0&1&1\\
				1&3&2&1\\
				2&17&8&8&1\\
				3&139&48&72&18&1\\
				4&1473&384&768&288&32&1\\
				5&19091&3840&9600&4800&800&50&1\\
				6&291793&46080&138240&86400&19200&1800&72&1\\
				7&5129307&645120&2257920&1693440&470400&58800&3528&98&1\\
				\hline
			\end{tabular}}

		\end{table}
\newpage
We observe that Lah numbers of type $B$ can be expressed by the falling factorial as follows:
\begin{align}\label{imp}
L_B(n,k) = \binom{n}{k} (2n+1)_{n-k}^B
\end{align}
where $(2n+1)_{n-k}^B$ is falling factorial of type $B$. As a consequence of Eq. (\ref{imp}), we can give the next result.

\begin{lem}\label{tog}
Let $n \geq 0$. Thus we have
\begin{itemize}
    \item [1.] $[x]^B_n =(x+2n)_n^B=\sum_{k=0}^{n} \binom{n}{k} (2n+1)_{n-k}^B(x)_{k}^B=\sum_{k=0}^{n} L_B(n,k) (x)_{k}^B$, 
    \item [2.] $(x)^B_n =\sum_{k=0}^{n} (-1)^{n-k} L_B(n,k) [x]_{k}^B$. 
\end{itemize}
\end{lem}	

\begin{proof}
It is clear from the definitions of the rising and the falling factorials of type $B$ that $[x]^B_n =(x+2n)_n^B$. Eq. (\ref{imp}) yields to the equality 
$$\sum_{k=0}^{n} L_B(n,k) (x)_{k}^B=\sum_{k=0}^{n} \binom{n}{k} (2n+1)_{n-k}^B(x)_{k}^B.$$
The signless Stirling numbers of the first kind in type $B$ actually appear as coefficients in the expression of the rising factorial of type $B$ in terms of ordinary powers $x^n$. Therefore, based on Eq. (\ref{50}) and Theorem \ref{m}, we conclude the following result: 
\begin{align*}
[x]^B_n &= \sum_{j=0}^{n} c_B(n,j) x^j = \sum_{j=0}^{n} c_B(n,j) \sum_{k=0}^{j} S_B(j,k)(x)_k^B \\
&= \sum_{k=0}^{n} \left(\sum_{j=k}^{n} c_B(n,j)S_B(j,k) \right)(x)_k^B
= \sum_{k=0}^{n} L_B(n,k) (x)_k^B.
\end{align*}
Considering the above facts together, we obtain the first part of the lemma. Similarly, the second part of the lemma can be easily proven by applying Eq. (\ref{9}) and Eq. (\ref{7}).
\end{proof}

\begin{rem}
Lemma \ref{tog} has the following additional interpretation. The set $V=K[x]$ be a vector space of all polynomials in the indeterminate $x$ with coefficients in the field $K$. The sets $B_1=\{1, [x]_1^B, [x]_2^B, \cdots\}$ and $B_2=\{1, (x)_1^B, (x)_2^B, \cdots\}$ are both bases of $V$. Then the first part of Lemma \ref{tog} asserts that the infinite matrix $\textbf{L}=[L_B(n,k)]_{k,n \in \mathbb{N}}$ is the transition matrix between the basis $B_2$ and the basis $B_1$. Conversely,  the second part of Lemma \ref{tog} asserts that the infinite matrix $\textbf{M}=[(-1)^{n-k}L_B(n,k)]_{k,n \in \mathbb{N}}$ is the transition matrix between the basis $B_1$ and the basis $B_2$. Therefore, the matrix $\textbf{M}$ is the inverse to the matrix $\textbf{L}$.
\end{rem}

The following theorem provides an important relationship between Cauchy numbers of both kinds in type $B$ by using Lah numbers $L_B(n,k)$. 
\begin{thm}\label{lahi}
Let $n \geq 0$, then 
\begin{align*}
c_n^B = \sum_{k=0}^{n} L_B(n,k) C_k^B \quad and \quad C_n^B = \sum_{k=0}^{n} (-1)^{n-k} L_B(n,k) c_k^B
\end{align*}
where $ L_B(n,k)$ is of type $B$ Lah numbers.
\end{thm}		
\begin{proof}
Consider the first part of Lemma \ref{tog}. We obtain by integrating of $[x]^B_n =\sum_{k=0}^{n} L_B(n,k) (x)_{k}^B$ from $0$ to $1$ on both sides that $ c_n^B = \sum_{k=0}^{n} L_B(n,k) C_k^B$.
Using the second part of Lemma \ref{tog}, it can be easily proved that $C_n^B = \sum_{k=0}^{n} (-1)^{n-k} L_B(n,k) c_k^B$ in a similar way.
\end{proof}

\begin{rem}
Lemma \ref{tog} and Theorem \ref{lahi} can both be thought of as a Lah inversion in type $B$ in the sense of 3.38 Corollary (iii) on page 96 of \cite{maigner1979}.   
\end{rem}

\section{Future Directions}
Two kinds of incomplete Cauchy numbers of type $B$, which are generalizations of the Cauchy numbers of type $B$, may be investigated in future studies by introducing the restricted Stirling numbers of type $B$. In addition, the identities associated with Cauchy numbers of both kinds in type $B$ can be generalized into colored type $G_{m,n}$ as a future work. Furthermore, new Cauchy numbers may be studied by considering the Stirling numbers of the second kind in type $D$.

\bibliographystyle{plain}

	\end{document}